\documentclass[12pt,a4paper,oneside]{amsart}

\usepackage{anysize}
\marginsize{2cm}{2cm}{1.5cm}{1.5cm}
\newtheorem{thm}{Theorem}[section]
\newtheorem{prp}[thm]{Proposition}
\newtheorem{lem}[thm]{Lemma}
\newtheorem{cor}[thm]{Corollary}

\theoremstyle{definition}
\newtheorem{dfn}{Definition}[section]
\newtheorem{rem}[dfn]{Remark}

\newcommand{\st}{:\;}

\def\R{{\mathbb R}}%

\newcommand{\Red}{\R^d}
\renewcommand{\phi}{\varphi}

\providecommand{\parenth}[1]{\left(#1\right)}%
\providecommand{\braces}[1]{\left\{#1\right\}}%
\newcommand{\iprod}[2]{\left\langle#1,#2\right\rangle}%
\def\polar{\circ}
\newcommand{\polarset}[1]{{#1}^{\polar}}%

\newcommand{\conv}{\mathrm{conv}}%

\newcommand{\pos}[1]{\mathrm{pos}\ \!#1}%
\newcommand{\lpos}[1]{\mathrm{lpos}\ \!#1}%

\newcommand{\enorm}[1]{\left|#1\right|}
\providecommand{\card}[1]{\lvert#1\rvert}%
\providecommand{\parenth}[1]{\left(#1\right)}%
\providecommand{\braces}[1]{\left\{#1\right\}}%
\newcommand{\Lin}[1]{\mathrm{Lin}\ \! {#1}}%

\usepackage{xcolor}
\usepackage{hyperref}
\hypersetup{
	colorlinks   = true, 
	urlcolor     = blue, 
	linkcolor    = blue, 
	citecolor    = red 
}
\newcommand{\Href}[2]{\hyperref[#2]{#1~\ref{#2}}}

\title{Colorful positive bases decomposition and Helly-type results for cones}
\author{Grigory Ivanov\address{Grigory Ivanov: 
Pontifícia Universidade Católica do Rio de Janeiro \\
Departamento de Matemática,
Rua Marquês de São Vicente, 225\\
Edif{\'i}cio Cardeal Leme, sala 862,
22451-900 G{\'a}vea, Rio de Janeiro, Brazil}
\email{grimivanov@gmail.com}}
\thanks{The author is supported by Projeto Paz and Coordena\c{c}\~ao de Aperfei\c{c}oamento de Pessoal de N\'ivel Superior -- Brasil (CAPES) -- 23038.015548/2016-06}


\begin{document}

\begin{abstract}
We prove the following colorful Helly-type result: Fix $k \in [d-1]$. Assume $\mathcal{A}_1, \dots, \mathcal{A}_{d+(d-k)+1}$ are finite sets (colors) of nonzero vectors in $\R^d$. If for every rainbow sub-selection $R$ from these sets of size at most $\max \{d+1, 2(d-k+1)\}$, the system $\iprod{a}{x} \leq 0,\; a \in R$ has at least $k$ linearly independent solutions, then at least one of the systems $\iprod{a}{x} \leq 0,\; a \in \mathcal{A}_i,$ $i \in [d+(d-k)+1]$ has at least $k$ linearly independent solutions.

A \emph{rainbow sub-selection} from several sets refers to choosing at most one element from each set (color).

The Helly number $\max \{d+1, 2(d-k+1)\}$ and the number of colors $d+(d-k)+1$ are optimal.

Our key observation is a certain colorful Carath\'eodory-type result for positive bases.
\end{abstract}

\maketitle

\section{Introduction}

The celebrated Helly's theorem \cite{helly1923mengen} posits that within a finite family of convex sets in $\R^d$, if the intersection of any subfamily, consisting of at most $d+1$ sets, shares a common point, then all the sets in the family share at least one common point. This gem of combinatorial convexity has seen many different extensions and generalizations (see \cite{barany2022helly} for one of the most up-to-date surveys). A now-classical result by Lov\'asz (see \cite{barany1982generalization}) is the following ``colorful'' version of Helly's theorem:

\begin{prp}[Colorful Helly theorem]\label{prp:colorful_Helly}
    Let $\mathcal{F}_1, \dots, \mathcal{F}_{d+1}$ be finite families of convex sets in $\R^d.$ Suppose that for any choice $F_1 \in \mathcal{F}_1, \dots, F_{d+1} \in \mathcal{F}_{d+1},$ the intersection $\bigcap\limits_{i=1}^{d+1} F_i$ is non-empty. Then for some $i \in \braces{1, \dots, d+1},$ the intersection of all sets in the family $\mathcal{F}_i$ is non-empty.
\end{prp}

Clearly, Lov\'asz's result implies the original Helly's theorem by taking $\mathcal{F}_1 = \dots = \mathcal{F}_{d+1} = \mathcal{F}$. 

Recently, B\'ar\'any \cite[Corollary 4.1]{barany2024positive} obtained the following Helly-type result, generalizing the result of Katchalski \cite{katchalski1978helly}:

\begin{prp}\label{prp:helly_hom_systems}
Assume $A$ is a finite set of nonzero vectors in $\R^d.$ The homogeneous system of linear inequalities $\iprod{a}{x} \leq 0, $ $a \in A$ has at least $k$ linearly independent solutions if and only if for every $B \subset A$ whose size is at most $\max \{d+1, 2(d-k+1)\},$ the system $\iprod{b}{x} \leq 0,$ $b \in B$ has at least $k$ linearly independent solutions.
\end{prp}

B\'ar\'any also showed that the constant $m(k,d)= \max \{d+1, 2(d-k+1)\}$ is optimal in the following sense. There are examples of finite subsets of $\R^d \setminus \braces{0}$ such that the solution set of any subsystem corresponding to a subset of at most $m(k,d) - 1$ points contains $k$ linearly independent solutions, but the solution set of the original system does not.

Suppose one has a finite collection of sets (colors). A \emph{rainbow sub-selection} from these sets refers to choosing at most one element from each set (color). We note that for convenience, we do not require a rainbow sub-selection to intersect each of the original sets. For the sake of convenience, we will use $[n]$ to denote the set $\{1, \dots, n\}$ for a natural $n.$

Our foremost result is a ``colorful'' extension of B\'ar\'any's result:

\begin{thm}\label{thm:colorful_linear_system_to_big} Fix $k \in [d-1]$. Assume $\mathcal{A}_1, \dots, \mathcal{A}_{d+(d-k)+1}$ are finite sets of nonzero vectors in $\R^d.$ If for every rainbow sub-selection $R$ from these sets of size at most $\max \{d+1, 2(d-k+1)\}$, the system $\iprod{a}{x} \leq 0,\; a \in R$ has at least $k$ linearly independent solutions, then at least one of the systems $\iprod{a}{x} \leq 0,\; a \in \mathcal{A}_i,$ $i \in [d+(d-k)+1]$ has at least $k$ linearly independent solutions.
\end{thm}

Clearly, this theorem implies \Href{Proposition}{prp:helly_hom_systems}, and the constant $m(k,d)$ is optimal since it is optimal for its ``monochromatic'' version - \Href{Proposition}{prp:helly_hom_systems}. Surprisingly, the number of color classes $d+(d-k)+1$ is optimal as well. We find this especially intriguing since usually the Helly numbers for ``colorful'' and ``monochromatic'' Helly-type results are the same. Note that for some quantitative colorful Helly-type results \cite{damasdi2021colorful}, the precise Helly number remains unknown.

Notably, when $d=k$, both \Href{Proposition}{prp:helly_hom_systems} and \Href{Theorem}{thm:colorful_linear_system_to_big} can be derived from the Colorful Helly theorem (\Href{Proposition}{prp:colorful_Helly}) with $m(d,d) = d+1.$ This observation is proposed as an exercise for the reader.

Our proof of \Href{Theorem}{thm:colorful_linear_system_to_big} mostly follows B\'ar\'any's proof of \Href{Proposition}{prp:helly_hom_systems} with a critical distinction lying in the application of what we term the ``Colorful Reay's decomposition.''

Let us recall another gem of combinatorial convexity, namely the ``Colorful Carath\'eodory theorem'' obtained by B\'ar\'any in \cite{barany1982generalization}:

\begin{prp}[Colorful Carath\'eodory theorem]
\label{prp:colorful_Caratheodory}
Assume $S_1, \dots, S_{d+1}$ are subsets of $\R^d.$ If a point $p$ belongs to the convex hull of each of the sets $S_1, \dots, S_{d+1},$ then there are points $s_1 \in S_1, \dots, s_{d+1} \in S_{d+1}$ such that $p$ belongs to the convex hull of $\braces{s_1, \dots, s_{d+1}}.$
\end{prp}

We introduce a key observation regarding positive bases: a ``colorful'' extension of Reay's theorem \cite[Theorem 2.6]{reay1965generalizations}. This can be seen as a Colorful Carath\'eodory theorem for positive bases.

First, we recall some notions. Let $S$ be a subset of a linear space. Its positive hull, denoted $\pos S,$ is the set of all finite linear combinations of elements of $S$ with non-negative coefficients. In other words, $\pos S = \bigcup\limits_{\lambda \geq 0} \lambda \, \conv S,$ where $\conv S$ stands for the convex hull of $S.$ A set $B$ forms a \emph{positive basis} of a finite-dimensional linear space $L$ if $L = \pos B$ and removing any element from $B$ changes the positive hull, meaning $L = \pos B$, yet $L$ differs from $\pos\! \parenth{B \setminus \{b\}}$ for every $b$ in $B$.

\begin{thm}[Colorful Reay's decomposition]
\label{thm:coloful_Reay_decomposition}
Fix $k \in [d].$ Let $d+k$ sets $S_1, \dots, S_{d+k} \subset \R^d$ be such that for every $i \in [d+k],$ $\pos S_i$ contains a $k$-dimensional linear subspace. Then there is a rainbow sub-selection $\mathcal{R}$ from the sets $S_1, \dots, S_{d+k}$ and its partition $\mathcal{R} = R_1 \cup \dots \cup R_m$ such that
\begin{enumerate}
\item $\card{R_i} \geq \card{R_{i+1}} \geq 2$ for every $i \in [m-1]$.
\item For every $i \in [m],$ $\pos\! \braces{R_1 \cup \cdots \cup R_i}$ is a linear subspace of $\R^d$ of dimension $\sum\limits_{j \in [i]}\card{R_j}- i,$ and $\bigcup\limits_{j \in [i]} R_j$ is its positive basis.
\item $\pos\mathcal{R} = \pos\!\! \braces{R_1 \cup \cdots \cup R_m}$ contains a $k$-dimensional linear subspace.
\end{enumerate}
\end{thm}

The number (of colors) $d+k$ in \Href{Theorem}{thm:coloful_Reay_decomposition} is optimal as follows from the next lemma.

\begin{lem}\label{lem:coloful_Reay_decomposition_bound}
For every $d \geq k \geq 1,$ there are $d+k-1$ sets $S_1, \dots, S_{d+k-1} \subset \R^d$ such that for every $i \in [d+k-1],$ $\pos S_i$ contains a $k$-dimensional linear subspace, but the positive cone of any rainbow sub-selection from $S_1, \dots, S_{d+k-1}$ does not contain a $k$-dimensional subspace.
\end{lem}

We believe that a weaker version of \Href{Theorem}{thm:coloful_Reay_decomposition} with the same bound $d+k$ on the number of sets can be derived from the Colorful Carath\'eodory theorem (\Href{Proposition}{prp:colorful_Caratheodory}). However, we will use the fact that the size of $R_i$ is a monotonically decreasing function of $i$ to prove \Href{Theorem}{thm:colorful_linear_system_to_big}.

The following linear algebra consequence of \Href{Theorem}{thm:coloful_Reay_decomposition} is simply a dual statement to \Href{Theorem}{thm:colorful_linear_system_to_big}, and we believe it might be of interest.

Let $A$ be a subset of $\R^d.$ The \emph{lineality space} of $\pos{A}$ is the set $\pos{A} \cap (-\pos{A})$, which is a linear subspace of $\mathbb{R}^d$. It is the unique maximal-dimensional subspace that $\pos{A}$ contains. We use $\lpos{A}$ to denote the lineality space of $\pos{A}.$

Set $h(k,d)=\max \{d+1, 2(k+1)\}$ and $m(k,d)=\max \{d+1, 2(d-k+1)\}.$ Note that $h(k,d) = m(d-k,d).$

\begin{thm}\label{thm:colorful_Caratheodory_for_lineality_subspace} Fix $k \in [d-1].$ Assume $A_1, \dots, A_{d+k+1}$ are subsets of $\R^d,$ and $\dim \lpos R \leq k$ for every rainbow sub-selection $R$ from these sets of size at most $h(k,d)$. Then for some index $i \in [d+k+1],$ $\dim \lpos A_i \leq k.$
\end{thm}

The author believes that there are deep interconnections between the geometry of positive bases and Quantitative Helly-type and Carathéodory-type results, in particular with the quantitative version of Steinitz’s theorem \cite{steinitz1913bedingt}. 
The latter asserts that if the origin in $\R^d$ is contained in the interior of the convex hull of a set $S$, then there exist at most $2d$ points of $S$ whose convex hull still contains the origin in its interior.  
Equivalently, Steinitz's theorem implies that the cardinality of a positive basis of $\R^d$ is at most $2d$.  
The author believes that the reformulation of the Quantitative Steinitz problem in the language of positive bases is missing in the literature:

\medskip
\noindent
\textbf{Quantitative Steinitz Problem for positive bases.} \emph{ 
Determine the largest number $r_d$ with the following property:  
whenever the Euclidean unit ball in $\R^d$ is contained in the convex hull of a set $S$, the set $S$ contains a subset $\{b_1,\dots,b_{k}\}$ with $ k \leq 2d$ such that  
\[
\{b_1 - r u,\dots,b_k - r u\} \quad\text{contains a positive basis of }\R^d
\]
for every unit vector $u$ and every $r \in [0, r_d).$}

\medskip

Moreover, this point of view was the main motivation for the currently best bound in the quantitative Steinitz problem \cite{ivanov2024steinitz, Ivanov2025sphericalsteinitz, ivanov_QST_polarity_2025}.  
As suggested in \cite{ivanov2022quantitative} and further developed in \cite{almendra2022quantitative}, the same phenomenon also appears in quantitative Helly-type theorems via polar duality.

The rest of the paper is organized as follows: In the next section, we establish our notation and introduce several key technical lemmas that will be utilized throughout the paper. Next, we will derive \Href{Theorem}{thm:colorful_Caratheodory_for_lineality_subspace} and \Href{Theorem}{thm:colorful_linear_system_to_big} from \Href{Theorem}{thm:coloful_Reay_decomposition} in \Href{Section}{sec:Colorful_Helly_lineality_and_dual}. In \Href{Section}{sec:coloful_Reay_decomposition}, we will obtain \Href{Theorem}{thm:coloful_Reay_decomposition}. It is the longest and most technical proof in the paper; it does not require any other results from the paper except for the well-known results discussed in the next section. In \Href{Section}{sec:optimality_helly_numbers}, we will prove \Href{Lemma}{lem:coloful_Reay_decomposition_bound} and show that it implies that the numbers of colors $d+(d-k)+1$ and $d+k+1$ in \Href{Theorem}{thm:colorful_linear_system_to_big} and \Href{Theorem}{thm:colorful_Caratheodory_for_lineality_subspace} are optimal. In the final section, we will obtain a weak nonhomogeneous version of \Href{Theorem}{thm:colorful_linear_system_to_big}.

\section{Notation, general properties of positive bases and cones}
\subsection{Notation}
 We use $[n]$ to denote the set $\{1, \dots, n\}$ for a natural $n.$  The inner product of two vectors $x$ and $y$ of 
 $\R^d$ is denoted by $\iprod{x}{y}.$ The \emph{linear hull} of a set $S$ is denoted by $\Lin{S}.$
 
 Recall that by a \emph{rainbow sub-selection} from several sets (colors), 
 we are referring to choosing at most one element from each set (color). 
By a \emph{rainbow selection} from several sets, we understand a choice of exactly one element from each set (color).
 
 Recall that $h(k,d)=\max \{d+1,2(k+1)\}$ and $m(k,d)=\max \{d+1,2(d-k+1)\}.$ 

 We recall that the \emph{polar} of a set $S \subset \Red$ is defined by
\[
\polarset{S} = \braces{x \in \Red \st \iprod{x}{s} \leq 1 \quad \text{for all} \quad s \in S}.
\]

The set $S^{\circ \circ}$ is called the \emph{bipolar} of $S.$
Recall  the bipolar theorem (\cite[Theorem 1.12.1]{PolBalEng}):
\begin{prp}[Bipolar theorem]
    \label{prp:bipolar theorem}
For any set $K \subset \R^d,$ the bipolar
$K^{\circ \circ}$ coincides with the closure of 
$\conv \! \parenth{K \cup \{0\}}.$   
\end{prp}

%

\subsection{Positive basis and minimal positive basis}
Let $S$ be a subset of a linear space. Its \emph{positive hull}, denoted $\pos S,$ 
is the set of all finite linear combinations of elements of $S$ with non-negative coefficients.  We say that a set $B$ is a \emph{positive basis} of a  finite-dimensional linear space $L$ if $L = \pos B,$ but $L \neq \pos \!\parenth{B \setminus\{b\}}$ for any $b \in B.$ 

Let $B$ be a positive basis of $L.$
We say that a linear subspace $L^\prime$ of $L$  is a \emph{positively spanned subspace with respect to} $B$ if $B \cap L^\prime$ is a positive basis of $L^\prime.$
Additionally, if $\dim L^\prime = n$ and $\card{B \cap L^\prime} = n+1,$ then we say that  $L^\prime$ is \emph{minimal} and $B \cap L^\prime$ is the minimal positive basis of $L^\prime.$
Clearly, the minimal positive basis of a non-trivial space consists of at least two vectors.

We will use next lemma, proven in \cite[Theorem 4.1]{davis1954theory}.
\begin{prp}\label{prp:basis_contains_minimal_subbasis}
Let $B$ be a positive basis of a finite-dimensional space $L.$
Then $B$ contains a minimal positive basis of some minimal linear subspace of $L.$
\end{prp}

\begin{rem}
    \label{rem:first_rainbow_set}
    The assertion of \Href{Theorem}{thm:coloful_Reay_decomposition} implies that $R_1$ is a minimal positive basis of its positive hull.
\end{rem}

\subsection{Convex cones}
A \emph{convex cone} $C$ in $\R^d$ is  a set of the form 
$a + \pos\! {A}$ for some $A \subset \R^d.$ In this case, we will call $a$ an \emph{apex} of $C.$ Note that, according to our definition, an apex is not necessarily unique.  Here and in the rest of  the paper,  we do not assume that the origin is   the apex of a cone!  By a \emph{non-trivial} cone we understand a cone that differs from 
$\{0\}.$

We say that a convex cone  with apex at $a$ is \emph{pointed} if the cone does not contain any line. Clearly, if a convex cone is pointed, then it has a unique apex. 

The following lemma is elementary and  follows from the separation lemma.
\begin{lem}
    \label{lem:dual_of_pointed_cone}
    Assume $C$ is a closed convex  cone in $\R^d.$ If $C$ is pointed, then $C^\circ$ is full-dimensional in $\R^d.$   
\end{lem}

\subsection{Lineality space}

The \emph{lineality space} of $\pos A$ is denoted by $\lpos A$. 
Surprisingly, although positive bases and lineality spaces are used extensively in several areas of mathematics, their basic properties are scattered throughout the literature (often without proofs). 
We will continue this tradition here, as we will need only two elementary facts.

We use $\oplus$ to denote the direct sum. 
Let $A$ be a set in $\R^d$, and denote $L = \lpos A$ and $L^\perp$ its orthogonal complement.  
As mentioned on page~65 of \cite{rockafellar1970convex}, one can express $\pos A$ as the direct sum of $L$ and $L^\perp \cap \pos A$ whenever $L$ is non-trivial.  
We will need the following slightly more general statement, which follows from standard linear-algebraic considerations:

\begin{prp}\label{prp:basis_and_orth_complement}
Let $L$ be  a non-trivial subspace of $\R^d$, and let $L^\perp$ be its orthogonal complement.
Let $P$ denote the orthogonal projection onto $L^{\perp}$.
Let $A$ be a subset of $\R^d$ such that $\lpos A \supset L$.
Then
\[
\pos A = L \oplus \pos P(A).
\]
In the special case $\lpos A = L$, the cone $\pos P(A)$ is pointed.
\end{prp}

Using the previous statement, one obtains the following:

\begin{prp}\label{prp:basis_of_lineality_subspace}
Let $A$ be a subset of $\R^d$ with a non-trivial lineality space $\lpos A$.
Then $A$ contains a positive basis of $\lpos A$.
\end{prp}


\section{Proofs of \Href{Theorem}{thm:colorful_Caratheodory_for_lineality_subspace} and \Href{Theorem}{thm:colorful_linear_system_to_big}}
\label{sec:Colorful_Helly_lineality_and_dual}

We start by showing  how to derive \Href{Theorem}{thm:colorful_Caratheodory_for_lineality_subspace} from \Href{Theorem}{thm:coloful_Reay_decomposition}.

We will use the following simple arithmetic lemma.  
Recall that $h(k,d)=\max \{d+1,2(k+1)\}.$
\begin{lem}\label{lem:h(k,d)_arithmetic}
Let $j,k,d$ be integers satisfying $1 \leq j-1 \leq k \leq d-1.$ Then
$\frac {jk}{j-1}+j \leq h(k,d).$
\end{lem}
\begin{proof}
The lemma was proved in \cite[Claim~3.1]{barany2024positive}. For the sake of completeness, we provide an alternative short computation:
\[ \frac {jk}{j-1}+j \leq 2 (k+1) 
\quad \Longleftrightarrow \quad 
j-2 \leq \frac{j-2}{j-1} k. \]
The latter inequality holds whenever $1 \le j-1 \le k$, which is exactly our assumption. The lemma follows since $2(k+1) \leq h(k,d).$
\end{proof}

Let us continue with the proof of \Href{Theorem}{thm:colorful_Caratheodory_for_lineality_subspace}. 
We essentially copy the proof given in \cite{barany2024positive} for the monochromatic version of the result.
\begin{proof}[Proof of \Href{Theorem}{thm:colorful_Caratheodory_for_lineality_subspace}]
Define $L_i = \lpos A_i,$ $i \in [d+k+1],$ 
and $t = \min\limits_{i \in [d+k+1]} \dim L_i.$ 
Assume that $t \geq k+1,$ otherwise, there is nothing to prove. 
By \Href{Theorem}{thm:coloful_Reay_decomposition},    there is a rainbow sub-selection $\mathcal{R}$ 
from the sets $A_1, \dots, A_{d+k+1}$, together with its partition $\mathcal{R} = R_1 \cup \cdots \cup R_m$ 
such that 
\begin{enumerate}
\item $\card{R_i} \geq \card{R_{i+1}} \geq 2$ for every $i \in [m-1]$.
\item  For every $i \in [m],$ $\pos\! \braces{R_1 \cup \cdots \cup R_i}$ is a linear subspace of $\Red$ of dimension  
$ \sum\limits_{j \in [i]}\card{R_j}- i,$ and $\bigcup\limits_{j \in [i]} R_j$ is its positive basis.
\item   $ \pos\!\! \braces{R_1 \cup \cdots \cup R_m}$ contains a $(k+1)$-dimensional linear subspace. 
\end{enumerate}
The last property means that  $R_1 \cup \dots \cup R_m$ is a positive basis of at least $(k+1)$-dimensional linear space $\pos\!\! \braces{R_1 \cup \cdots \cup R_m}.$

We now prove by induction on $j$ that $ \dim \pos \! \braces{R_1\cup \ldots \cup R_j} \leq k.$  

The base case $j=1.$ 
Since $R_1$ is a subset of $\R^d,$ $\dim \pos{R_1} \leq d.$ As $R_1$ is a minimal positive basis of $\pos{R_1},$ we have $\card{R_1} \leq d+1   \leq h(k,d).$ By the assumption of the theorem, $\dim{  \lpos R_1} \leq k.$

For the inductive step $j-1 \to j$, assume that $j\geq 2.$ Since $R_j$ has the smallest size among the sets $R_1, \ldots, R_j,$ we have 
\begin{equation}
\label{eq:last_rainbow_size}
\card{R_j}\leq \frac{1}{j-1} \sum\limits_{i \in [j-1]}\card{R_i}.
\end{equation} 
By the choice of the sets $R_1,\dots,R_{j-1}$ and by the induction hypothesis,
\begin{equation}
\label{eq:sum_card_rainbow}
\sum\limits_{i \in [j-1]}\card{R_i} =  \dim \pos\! \braces{R_1 \cup \cdots \cup R_{j-1}} + (j-1) \leq k+(j-1).
\end{equation}

Since each $\card{R_i} \ge 2$, we have
$\dim \pos \! \braces{R_1 \cup \dots \cup R_j} \ge j$ for every $j \in [m]$.
Therefore,
\begin{eqnarray*}
j \leq  \dim \pos\! \braces{R_1 \cup \cdots \cup R_{j}} = \sum\limits_{i \in [j-1]}\card{R_i} +\card{R_j} - j\\
    \stackrel{\eqref{eq:sum_card_rainbow} \text{ and } \eqref{eq:last_rainbow_size}}{\leq}  k+(j-1)+\frac {k+(j-1)}{j-1}-j=\frac {jk}{j-1}.
\end{eqnarray*}
Consequently, $j\leq \frac {jk}{j-1},$ which implies that $j-1\le k.$

We can now apply
 \Href{Lemma}{lem:h(k,d)_arithmetic}. It implies that  
\[
\sum\limits_{i \in [j]}\card{R_j} = 
\dim \pos\! \braces{R_1 \cup \cdots \cup R_{j}} + j   
\leq h(k,d).
\]
 Thus, 
by the assumption of the theorem, $\dim \pos\! \braces{R_1 \cup \cdots \cup R_{j}} \leq k.$ 

Finally, for $j=m,$ we get  $k+1 \leq \dim \pos\!\! \braces{R_1 \cup \cdots \cup R_m} \leq k. $ This contradiction shows that our assumption was false. We conclude $t \leq k,$ completing the proof of  \Href{Theorem}{thm:colorful_Caratheodory_for_lineality_subspace}.
\end{proof}

To prove \Href{Theorem}{thm:colorful_linear_system_to_big}, we need the following simple technical lemma.

\begin{lem}\label{lem:sol_set_via_polar_of_cone}
Assume $A$ is a  finite set of nonzero vectors in $\R^d.$  
We  use $L^\perp$ to denote the orthogonal complement of the lineality space
$\lpos A,$ and $P$ to denote the orthogonal projection onto $L^\perp.$
Denote the solution set of the homogeneous system $\iprod{a}{x} \leq 0, a \in A$ by $S.$
Then $S$ is a full-dimensional cone within $L^\perp,$ and the identities
\[
S = \polarset{\parenth{\pos A}} = L^\perp \cap \polarset{\parenth{\pos P(A)}}
\] 
hold.
\end{lem}
\begin{proof}
The leftmost identity $S = \polarset{\parenth{\pos A}}$ immediately follows from the definition of the polar set. 

By \Href{Proposition}{prp:basis_and_orth_complement}, 
$\pos A = \lpos A \oplus \pos P(A),$ and $\pos P(A)$ is pointed within $\R^d,$ and consequently, within $L^\perp$. Thus, by the standard properties of the polar transform ( see, for example, \cite[Theorem 1.6.9]{schneider2014convex}),
\[
\polarset{\parenth{\pos A}} = \polarset{\parenth{\lpos A \oplus \pos P(A)}} = L^\perp \cap \polarset{\parenth{\pos P(A)}}.
\] 
Since the cone $\pos P(A)$ is a pointed cone within $L^\perp,$ 
\Href{Lemma}{lem:dual_of_pointed_cone} implies that 
the cone $L^\perp \cap \polarset{\parenth{\pos P(A)}}$ is full-dimensional within 
$L^\perp.$ 
\end{proof}
\begin{proof}[Proof of \Href{Theorem}{thm:colorful_linear_system_to_big}]
By \Href{Lemma}{lem:sol_set_via_polar_of_cone},  the system $\iprod{a}{x} \leq 0, a \in \mathcal{A}_i$ has at least $k$ linearly independent solutions if and  only if 
the dimension of $\parenth{\lpos \mathcal{A}_i}^\perp$ is at least $k.$ Or equivalently, the dimension of $\lpos \mathcal{A}_i$ is at most $d-k.$  Similarly, 
for a rainbow sub-selection $R $  
from the sets  $\mathcal{A}_1, \dots, \mathcal{A}_{d+ (d-k)+1},$  the system $\iprod{a}{x} \leq 0,\; a\in R$ has at least $k$ linearly independent solutions if and only if the dimension of $\lpos R$ is at most $d-k.$ 
The desired implication follows from \Href{Theorem}{thm:colorful_Caratheodory_for_lineality_subspace}.
\end{proof}

\section{Proof of \Href{Theorem}{thm:coloful_Reay_decomposition}}
\label{sec:coloful_Reay_decomposition}
We start by reducing the problem to an easier one. 
\begin{thm}
\label{thm:coloful_Reay_decomposition_weak}
Fix $k \in [d].$ Let $d+ k$ sets $S_1, \dots, S_{d+k} \subset \R^d$ be such that
for every $i \in [d+k],$ 
$\pos S_i $ contains a $k$-dimensional linear subspace. 
Then there is a rainbow sub-selection $R_1 \cup \dots \cup R_m$ from the sets such that 
\begin{enumerate}
\item $\card{R_i} \geq \card{R_{i+1}} \geq 2$ for all $i \in [m-1]$.
\item  For every $i \in [m],$ $\pos\! \braces{R_1 \cup \cdots \cup R_i}$ is a linear subspace of $\Red$ of dimension  
$ \sum\limits_{j \in [i]}\card{R_j}- i.$ 
\item   $ \pos\!\! \braces{R_1 \cup \cdots \cup R_m}$ contains a $k$-dimensional linear subspace. 
\end{enumerate} 
\end{thm}
The only difference between \Href{Theorem}{thm:coloful_Reay_decomposition} and 
\Href{Theorem}{thm:coloful_Reay_decomposition_weak} is that in the latter  we do not state that $\bigcup\limits_{j \in [i]} R_j$ is a positive basis of 
$\pos\! \braces{R_1 \cup \cdots \cup R_i},$ $i \in [m].$ However, it is easy to see that \Href{Theorem}{thm:coloful_Reay_decomposition_weak} implies \Href{Theorem}{thm:coloful_Reay_decomposition}.

\begin{proof}[Proof of \Href{Theorem}{thm:coloful_Reay_decomposition}]
By \Href{Theorem}{thm:coloful_Reay_decomposition_weak}, there is  a rainbow sub-selection $R^\prime $ from  $S_1, \dots,  S_{d+k}$ such that
$\pos R^\prime$ is a linear subspace and $\dim \pos R^\prime \geq k.$ By \Href{Proposition}{prp:basis_of_lineality_subspace},  $R^\prime$ contains a positive
basis, say $R,$ of $\pos R^\prime.$ Applying   \Href{Theorem}{thm:coloful_Reay_decomposition_weak} to $d + \dim \pos R^\prime$ copies of 
$R,$ 
 we get the desired set system.
\end{proof}

From now on, we proceed with the proof of \Href{Theorem}{thm:coloful_Reay_decomposition_weak}.

Our proof is based on an inductive construction.  First, we  show how to choose the first rainbow set $R_1,$ then we will show how to find a next rainbow set $R_m$ provided that the dimension of 
$\lpos \!\braces{R_1 \cup \dots \cup R_{m-1}}$ is strictly less than $k.$ Finally, we will show that the process terminates after a finite number of steps and the constructed
sets satisfy the desired properties. 
\subsection{Choice of the first rainbow set}
\label{sub:first_rainbow_set}
We start with the simplest case of $k=1.$

\begin{lem}\label{lem:coloful_pos_bases_helly_k=1}
Assume $L$ is an $n$-dimensional vector space.
Assume that for every $i \in [n+1],$ a finite set $B_i$ 
is a positive basis of a non-trivial subspace $H_i$ of $L.$
Then there is a rainbow  selection $R$ from the sets $B_1, \dots, B_{n+1}$ 
such that  $ \lpos R $ is a non-trivial subspace of $L.$
\end{lem}
\begin{proof}
Fix $i \in [n+1].$ Since $B_i$ is a positive basis of $H_i,$ $0 \in \conv{B_i}.$ Thus, by the colorful Carath\'eodory theorem  -- \Href{Proposition}{prp:colorful_Caratheodory}, 
there is a rainbow selection $R$ such that 
$0 \in \conv{R}.$ Since the original sets are positive bases in the corresponding subspaces, they do not contain the origin. Thus, the origin belongs to the relative interior of a face $F$ of $\conv{R}$  satisfying $\dim F \geq 1.$ It means that
$ \lpos R \supset \Lin{F}.$ The lemma follows. 
\end{proof}

Using \Href{Proposition}{prp:basis_of_lineality_subspace}, \Href{Proposition}{prp:basis_contains_minimal_subbasis}, and 
\Href{Lemma}{lem:coloful_pos_bases_helly_k=1} together, we get
\begin{cor}\label{cor:existence_rainbow_basis}
Assume $L$ is an $n$-dimensional vector space.
Assume that for every $i \in [n+1],$ a set $S_i$ satisfies 
$\dim  \lpos{S_i} \geq 1.$
Then there is a non-empty rainbow  sub-selection $R$ from the sets  $S_1, \dots, S_{n+1}$  such that  $R$ is a minimal basis of the minimal  subspace $\pos R.$ 
\end{cor}

By \Href{Corollary}{cor:existence_rainbow_basis}, there is a non-empty subset $I$ of $[d+k]$  and a rainbow sub-selection $R$ from $\bigcup\limits_{i \in I} S_i$ such that  $R$ is a minimal positive basis of the minimal  subspace $\pos R.$ 
Let $I_1 \subset [d+k]$   be an index set of maximal cardinality for which there is   a rainbow sub-selection $R_1$  from the union  $\bigcup\limits_{i \in I_1} S_i$ such that $R_1$ is a minimal positive basis of the minimal subspace $\pos{R_1}.$ 

Note that automatically: 
\begin{equation}\label{eq:first_step}
\card{R_1} \geq 2,\ \pos R_1 = \lpos R_1, \quad \text{and} \quad 
\dim \pos R_1 = \card{R_1} -1.
\end{equation}


\subsection{Inductive construction}
\label{subsec:next_rainbow_set}
Assume that for some $m^\prime > 1,$ the sets $R_1, \dots, R_{m^\prime-1}$ have already been chosen.

If the dimension of $\lpos{\!\braces{R_1 \cup \dots \cup R_{m^\prime-1}}}$ is at least $k,$ we terminate the construction process.

 If the dimension of $\lpos{\!\braces{R_1 \cup \dots \cup R_{m^\prime-1}}}$ is strictly less than $k,$ we proceed as follows.

Denote $\lpos{\!\!\braces{R_1 \cup \dots \cup R_{m^\prime-1}}}$ by $L_{m^\prime-1}.$
We use $C_{m^\prime-1}$ to denote the orthogonal complement of $L_{m^\prime-1},$ and $P_{m^\prime-1}$ to denote the orthogonal projection onto $C_{m^\prime-1}.$ 
We set 
$I_{m^\prime} \subset [k+d] \setminus 
\parenth{I_1 \cup \dots \cup I_{m^\prime-1}}$ 
to be a non-empty index set of maximal cardinality such that  there is  a rainbow sub-selection $R_{m^\prime}$ from the sets
  $S_i, i \in I_{m^\prime},$
  such that
$P_{m^\prime-1} \ \!\!\! \parenth{R_{m^\prime}}$ is a minimal positive basis  of the minimal linear subspace $\pos{P_{m^\prime-1}\ \!\!\! \parenth{R_{m^\prime}}}.$ 

\subsection{Correctness of the inductive construction}
We divide the proof into several simple steps.

Now we show that it is possible to add a new rainbow set if the dimension of the lineality space of the union of already chosen ones is strictly less than $k.$

\begin{lem}\label{lem:existence_of_the_rainbow_next_set}
Assume that for some $m^\prime > 1,$ the sets $R_1, \dots, R_{m^\prime-1}$ have
already been  chosen in such a way that
\begin{enumerate}
\item 
 $ \dim \lpos{\!\!\braces{R_1 \cup \dots \cup R_{m^\prime-1}}} < k.$
\item $\card{R_i} \geq  2$ for each $i \in [m^\prime-1]$.
\item  For every $i \in [m^\prime-1],$ $\dim \lpos\! \!\parenth{R_1 \cup \cdots \cup R_i} =  \sum\limits_{j \in [i]}\card{R_j}- i.$ 
\end{enumerate}
 Then there is an index set $I \subset [k+d] \setminus 
\parenth{I_1 \cup \dots \cup I_{m^\prime-1}}$ and a rainbow sub-selection 
 $R$ from the sets $S_i, i \in I,$ such that
$P_{m^\prime-1} \ \!\!\! \parenth{R}$ is a minimal positive basis  of the minimal linear subspace $\pos{P_{m^\prime-1}\ \!\!\! \parenth{R}}.$
\end{lem}
\begin{proof}
For any $i,$ the dimension of the space $\lpos{P_{m^\prime-1}(\lpos S_i)}$ is at least
$\dim \lpos S_i - \dim \lpos{\!\!\braces{R_1 \cup \dots \cup R_{m^\prime-1}}} \geq 1.$
That  is, $\lpos{P_{m^\prime-1}(\lpos S_i)}$ is a non-trivial subspace of
$C_{m^\prime-1}.$ 
On the other hand, the cardinality of the set $[k+d] \setminus 
\parenth{I_1 \cup \dots \cup I_{m^\prime-1}}$ equals 
\[
d+k - \sum\limits_{j \in [m^\prime -1]}\card{R_j} = d - \parenth{\sum\limits_{j \in [m^\prime -1]}\card{R_j} - (m^\prime -1)} + (k - (m^\prime -1)).
\]
By our assumption, 
  the cardinality of the set $[k+d] \setminus 
\parenth{I_1 \cup \dots \cup I_{m^\prime-1}}$ equals
\begin{equation}
\label{eq:card_remaining_index_set}
 d - \dim \lpos{\!\!\braces{R_1 \cup \dots \cup R_{m^\prime-1}}} + (k - (m^\prime -1)).
\end{equation}

Since $\card{R_i} \geq 2$ for all $i \in [m^\prime-1],$ we get  
\[k > \dim \lpos{\!\!\braces{R_1 \cup \dots \cup R_{m^\prime-1}}} = 
\sum\limits_{j \in [m^\prime -1]}\card{R_j} - (m^\prime -1) \geq  m^\prime -1.
\]
Hence, 
$k - (m^\prime -1) \geq 1.$ However, 
 $ \dim C_{m^\prime-1} = d - \dim\lpos{\!\!\braces{R_1 \cup \dots \cup R_{m^\prime-1}}}.$ 
Thus, using \eqref{eq:card_remaining_index_set} in \Href{Corollary}{cor:existence_rainbow_basis}, we conclude that there is a non-empty subset $I$ of $[k+d] \setminus 
\parenth{I_1 \cup \dots \cup I_{m^\prime-1}}$  and a rainbow sub-selection  $R^\prime$ from the sets $P_{m^\prime-1} (S_i), i \in I,$  such that 
$R^\prime$ is a minimal positive basis  of the minimal linear subspace 
$\lpos{R^\prime}.$ There are $s_i \in S_i,$ $i \in I$ such that 
$R^\prime =  \bigcup\limits_{i \in I} P_{m^\prime-1} (s_i).$ We complete the proof by taking $R  =  \bigcup\limits_{i \in I} \braces{s_i}.$
\end{proof}

\begin{cor}\label{cor:existence_of_the_next_rainbow_set}
Assume that for some $m^\prime > 1,$ the sets $R_1, \dots, R_{m^\prime-1}$ have  already been chosen by our inductive construction in such a way that
 $\lpos{\!\!\braces{R_1 \cup \dots \cup R_{m^\prime-1}}} < k.$ Then there are
 $I_{m^\prime} \subset [k+d] \setminus 
\parenth{I_1 \cup \dots \cup I_{m^\prime-1}} $ of maximal cardinality and a rainbow sub-selection  $R_{m^\prime}$ from $ S_i, i \in I_{m^\prime},$ 
such that $P_{m^\prime-1} \ \!\!\! \parenth{R_{m^\prime}}$ is a minimal positive basis  of the minimal linear subspace $\pos{P_{m^\prime-1}\ \!\!\! \parenth{R_{m^\prime}}}.$ 
\end{cor}

Now let us show that the other desired properties of $R_{m^\prime}$ hold.

We will need the following lemma, which is essentially Lemma 2.5 of \cite{reay1965generalizations}.

\begin{lem}\label{lem:Reay_property_of_iteration_step}
 Let $S $ be  a  subset of a  linear space. Assume 
\begin{enumerate}
\item There is  a subset  $R$  of $S$ of maximal cardinality 
 such that $R$ is a minimal positive basis of the minimal  subspace $\pos{R}.$
\item There is  a subset  $Q$  of $S \setminus R$ of maximal cardinality 
 such that  $P\ \!\!\! \parenth{Q}$ is a minimal positive basis of the minimal subspace $\pos{P\!\parenth{Q}},$ where $P$ denotes
the orthogonal projection onto the orthogonal complement of $\pos{R}.$
\end{enumerate} 
Then
$\card{Q} \leq \card{R},$ 
$\pos\!\!\braces{R \cup Q} = \lpos\!\!\braces{R \cup Q} = \pos\!\!\braces{R \cup P(Q)} ,$
and   $\dim \pos\!\!\braces{R \cup Q} = \card{Q} +  \card{R} - 2.$
\end{lem}
\begin{proof}
Note that $R$ and $Q$ are disjoint since $P(R) = 0$ and $P(Q)$ is a positive basis of a subspace of the image of $P.$  Denote $I = \left[ \card{R} \right]$ and 
$J = \left[ \card{Q} \right].$ Let $R = \{r_1, \dots, r_{\card{R}}\}$ and 
$Q = \{q_1, \dots, q_{\card{Q}}\}.$
Since  $P\ \!\!\! \parenth{Q}$ is a minimal positive basis of the minimal subspace $\pos{P\!\parenth{Q}},$ there is a unique positive convex combination 
$\sum_{j \in J} \beta_j P(q_j)$ that equals zero. Since $P$ is the projection onto the orthogonal complement of $\pos R,$ one has $\sum_{j \in J} \beta_j q_j \in \pos R.$  

Assume $\sum_{j \in J} \beta_j q_j = 0.$ Then the origin is in the relative interior of  $ \conv\ \! Q,$ and therefore $\pos\ \! Q$ is a linear subspace of dimension $\card{Q} - 1.$ In this case the choice of $R$ implies the inequality 
$\card{R} \geq \card{Q}.$ Moreover, in this case, the subspace $\pos \braces{R \cup Q}$ is equal to the direct sum of the linear subspaces $\pos R$ and $\pos Q.$ Hence,
$\dim \pos \! \braces{R \cup Q} = \card{R} +  \card{Q} - 2$ and $R \cup Q$ is a positive basis of $\pos \! \braces{R \cup Q}.$

Assume now that $\sum_{j \in J} \beta_j q_j \neq  0.$ 
Since $R$ is a minimal basis of $\pos R,$  there is a unique  proper  subset $I^\prime$ of $I$ of minimal cardinality and positive coefficients $\alpha_i, i \in I^\prime,$ such that 
\[
\sum_{j \in J} \beta_j q_j + \sum_{i \in I^\prime} \alpha_i r_i = 0.
\]
This equation and the minimality of $\card{I^\prime}$ imply that the origin is in the relative interior of $\conv \parenth{R^\prime \cup Q},$ where $R^\prime = \{r_i  \st i \in I^\prime \}.$
Hence, $\pos \!\! \braces{R^\prime \cup Q} = \Lin{\ \!\!\!\braces{R^\prime \cup Q}}.$
However, the construction shows that 
$\dim \parenth{R^\prime \cup Q} = \card{I^\prime}+ \parenth{\card{Q}  - 1}.$ Hence $R^\prime \cup Q$ is the minimal positive basis for  
$ \pos \!\braces{R^\prime \cup Q}.$ By the choice of $R,$ 
$\card{R} \geq \card{ Q \cup R^\prime} > \card{Q}.$

Since $\pos R = \lpos R$ by the choice of $R,$ we have
\[
\pos \!\braces{R \cup Q} = \pos \! \braces{R \cup P(Q)}.
\]
Clearly, $\card{P(Q)} = \card{Q}.$ However, $\pos \! \!\braces{R\cup P(Q)}$ is the direct sum of $\pos R$ and $\pos P(Q),$ which is a linear subspace of dimension $\card{R} + \card{Q} -2$ by the choice of $R$ and $Q.$ 
 The proof of 
\Href{Lemma}{lem:Reay_property_of_iteration_step} is complete.
\end{proof}

\begin{lem}\label{lem:properties_of_the_rainbow_next_set}
Assume that for some $m^\prime > 1,$ the sets $R_1, \dots, R_{m^\prime-1}$ have already been  chosen by our inductive construction in such a way that
\begin{enumerate}
\item 
 $\dim \lpos{\!\!\braces{R_1 \cup \dots \cup R_{m^\prime-1}}} < k.$
\item $\card{R_i} \geq  2$ for each $i \in [m^\prime-1]$.
\item  For every $i \in [m^\prime -1],$ $\pos\!\! \braces{R_1 \cup \cdots \cup R_i}$ is a linear subspace of $\Red$ of dimension  
$ \sum\limits_{j \in [i]}\card{R_j}- i.$ 
\end{enumerate}
 Then the inductive construction described in \Href{Subsection}{subsec:next_rainbow_set} returns  $R_{m^\prime}$ satisfying the following properties:
 \begin{enumerate}
\item $\card{R_{m^\prime}} \geq  2.$
\item $\card{R_{m^\prime-1}} \geq \card{R_{m^\prime}}.$
\item  $\pos\! \!\braces{R_1 \cup \cdots \cup R_{m^\prime}}$ is a linear subspace of $\Red$ of dimension  
$ \sum\limits_{j \in [m^\prime]}\card{R_{j}}- m^\prime.$ 
\end{enumerate}
\end{lem}
\begin{proof}

By \Href{Corollary}{cor:existence_of_the_next_rainbow_set}, the inductive construction described in \Href{Subsection}{subsec:next_rainbow_set} returns  $R_{m^\prime}.$ By construction, $\card{R_{m^\prime}} \geq 2.$ 

We use $P_0$ to denote the identity operator on $\R^d,$ $L_0$ and $C_0$ to denote 
$0$ and $\R^d,$ respectively.  
Set
$R = P_{m^\prime -2} (R_{m^\prime -1})$  and $Q = P_{m^\prime -2} (R_{m^\prime}).$ 
Since $C_{m^\prime - 2} \supset C_{m^\prime - 1} \supset C_{m^\prime },$ 
\[P_{m^\prime -1}(Q) = P_{m^\prime -1}P_{m^\prime -2}(R_{m^\prime}) = P_{m^\prime -1} (R_{m^\prime}). 
\] 
Hence, 
$\card{R} = \card{R_{m^\prime -1}}$ and  $\card{Q} = \card{R_{m^\prime}}.$ 
Moreover, $R$ is a minimal basis of  $\pos R$ by the choice of $R_{m^\prime -1}.$
Thus, the sets $R$ and $Q$ satisfy the assumptions of \Href{Lemma}{lem:Reay_property_of_iteration_step} with $S = R \cup Q$. 
Consequently, $\card{R_{m^\prime -1}} \geq \card{R_{m^\prime}}$ and 
$\pos \!\!\braces{R \cup Q} $ is a linear subspace of  dimension $\card{R_{m^\prime -1}} + \card{R_{m^\prime}} -2.$  

The lemma follows in the case $m^\prime =2.$ If $m^\prime > 2,$ 
$L_{m^\prime -2}$ is non-trivial.
By \Href{Proposition}{prp:basis_and_orth_complement}, applied with 
$L = L_{m^\prime -2}$ and $P = P_{m^\prime -2}$, we obtain
\[
\pos L_{m^\prime}
= \pos\!\bigl(L_{m^\prime -2} \cup R_{m^\prime -1} \cup R_{m^\prime}\bigr)
= \pos L_{m^\prime -2} \;\oplus\;
\pos\!\bigl( P_{m^\prime -2}(R_{m^\prime -1}) 
              \cup 
              P_{m^\prime -2}(R_{m^\prime}) \bigr).
\]
Since $R = P_{m^\prime -2}(R_{m^\prime -1})$ and 
$Q = P_{m^\prime -2}(R_{m^\prime})$, the above becomes
\[
\pos L_{m^\prime}
= \pos L_{m^\prime -2} \;\oplus\; \pos\!\braces{R \cup Q}.
\]
Thus,
\[
\dim \pos L_{m^\prime}
= \dim \pos L_{m^\prime -2}
  + \bigl(\card{R_{m^\prime -1}} + \card{R_{m^\prime}} - 2\bigr).
\]
Using the induction hypothesis,
\[
\dim \pos L_{m^\prime -2}
= \sum_{j \in [m^\prime -2]} \card{R_j} - (m^\prime -2),
\]
we obtain
\[
\dim \pos L_{m^\prime}
= \sum_{j \in [m^\prime -2]} \card{R_j} - (m^\prime -2)
  + \card{R_{m^\prime -1}} + \card{R_{m^\prime}} - 2
= \sum_{j \in [m^\prime]} \card{R_j} - m^\prime.
\]

\end{proof}

Let us combine together all the small lemmas to prove \Href{Theorem}{thm:coloful_Reay_decomposition_weak}.
\begin{proof}[Proof of \Href{Theorem}{thm:coloful_Reay_decomposition_weak}]
Choose $R_1$ as described in \Href{Subsection}{sub:first_rainbow_set}.
The existence of such a choice was justified in \Href{Subsection}{sub:first_rainbow_set}.

 If $\card{R_1} > k,$ then the inductive construction
terminates immediately at $m = 1$ and by \eqref{eq:first_step} $R_1$ is the desired rainbow sub-selection.  

If $\card{R_1} \leq k,$ we use the inductive construction from \Href{Subsection}{subsec:next_rainbow_set} to construct sets $R_2, \dots, R_m$ while
$\dim \lpos \! \braces{R_1 \cup \dots \cup R_{m-1}}  < k.$ 
By \Href{Corollary}{cor:existence_of_the_next_rainbow_set} and \eqref{eq:first_step}, we can choose the next rainbow set as described in the inductive construction. By \Href{Lemma}{lem:properties_of_the_rainbow_next_set},
the constructed set system satisfies the desired properties. Since we add at least one dimension at each step, the process terminates after at most $k$ steps.  
\end{proof}

\section{Optimality of Helly numbers}
\label{sec:optimality_helly_numbers}
We start by proving \Href{Lemma}{lem:coloful_Reay_decomposition_bound}, which shows the optimality of $d+k$ in \Href{Theorem}{thm:coloful_Reay_decomposition}.

\begin{proof}[Proof of \Href{Lemma}{lem:coloful_Reay_decomposition_bound}]
We proceed by induction on $d$.

For $d=k$, set 
\[
S_1 = \dots = S_{2k-1} = \braces{\pm e_1, \dots, \pm e_k}.
\]
Any rainbow sub-selection $R$ from these sets omits at least one of the $2k$ points $\braces{\pm e_1, \dots, \pm e_k}$. Hence, $\pos R$ cannot contain the whole space $\R^k$.

Now we explain how to add one dimension and one new set to our set system.  
We assume that $\R^{k}$ is the subspace of $\R^d$ spanned by its first $k$ standard basis vectors $e_1, \dots, e_k$ for any $k \in [d]$.

The inductive construction is as follows: assume that for some $d>k$ we have already constructed sets
\[
S_1, \dots, S_{d+k-2} \subset \R^{d-1}.
\]
We then define $S_{d+k-1}$ to be any positive basis of $\R^{d}$ whose vectors do not lie in $\R^{d-1}$.

We now show that this construction indeed produces the desired sets. More precisely, we will show that the lineality space of any rainbow sub-selection $R$ coincides with the lineality space of the elements of $R$ taken from the first $2k-1$ sets, that is,
\[
\lpos R
  = 
  \lpos\!\left( R \cap \left( S_1 \cup \dots \cup S_{2k-1} \right) \right).
\]
Since $\pos S_1 = \dots = \pos S_{2k-1} = \R^{k}$ and $\pos S_{2k-1+i} = \R^{k+i}$, and since we have already settled the case $d=k$, this identity of lineality spaces implies the desired bound
\[
\dim \lpos R < k.
\]

Assume now that $\lpos R$ is non-trivial; otherwise, there is nothing to prove.  
By \Href{Proposition}{prp:basis_of_lineality_subspace}, the set $R$ contains a positive basis $F$ of its lineality space.  
Let $j$ be the largest index such that $S_j \cap F$ is non-empty.  
If $j \le 2k-1$, there is again nothing to prove.  
If $j > 2k-1$, then the single element of $S_j \cap F$ has a non-zero $(j-(k-1))$-th coordinate, and it is the only element of $F$ whose $(j-(k-1))$-th coordinate is non-zero.  
This contradicts the fact that $F$ is a positive basis of a non-trivial subspace.

The proof of \Href{Lemma}{lem:coloful_Reay_decomposition_bound} is complete.
\end{proof}

\begin{rem}
The sets $S_i, i > 2k-1,$ can be chosen in such a way that $\dim\lpos S_i = k$.
\end{rem}

\begin{cor}\label{cor:optimality_lineality_subspace}
The number of colors $d + k + 1$ in \Href{Theorem}{thm:colorful_Caratheodory_for_lineality_subspace} is optimal. For any $k \in [d-1],$ there are subsets $A_1, \dots, A_{d+k}$ of $\R^d$  with $\dim \lpos A_i > k,$ $i \in [d+k],$ and such that   $\dim \lpos R \leq k$ for every rainbow sub-selection $R$ from these sets.
\end{cor}

\begin{rem}
We do not assume that the size of $R$ is at most $h(k,d)$. 
\end{rem}

\begin{proof}
Since $k \in [d-1],$ $1 < k+1 \leq d$. Thus, any $d+k$ sets returned by \Href{Lemma}{lem:coloful_Reay_decomposition_bound} for $k \to k+1$ provide the desired example.
\end{proof}

Similarly, one gets

\begin{cor}
The number of colors $d + (d-k) + 1$ in \Href{Theorem}{thm:colorful_linear_system_to_big} is optimal. For any $k \in [d-1],$ there are $\mathcal{A}_1, \dots, \mathcal{A}_{d+ (d-k)}$ finite subsets of nonzero vectors in $\R^d$ such that the dimension of the solution set of the systems $\iprod{a}{x}\leq 0,\; a\in \mathcal{A}_i,$ $i \in [d+ (d-k)]$ is strictly less than $k$, but for every rainbow sub-selection $R$ from these sets the system $\iprod{a}{x} \leq 0,\; a\in R$ has at least $k$ linearly independent solutions.
\end{cor}

\begin{proof}
Since $k \in [d-1],$ $1 < d-k+1 \leq d$. By \Href{Lemma}{lem:sol_set_via_polar_of_cone}, any $d+(d-k)$ sets returned by \Href{Corollary}{cor:optimality_lineality_subspace} for $k \to d-k$ provide the desired example.
\end{proof}

As mentioned in the Introduction, the optimality of $h(k,d)$ and $m(k,d)$ was shown in \cite{barany2024positive}. For completeness, we provide the example that shows the optimality of $h(k,d)$ (the optimality of $m(k,d)$ follows immediately from it and \Href{Lemma}{lem:sol_set_via_polar_of_cone}).

Consider the case $d+1 \geq 2(k+1)$ in which $h(k,d) = d+1$. Set $A$ to be the vertices $v_1, \dots, v_{d+1}$ of a regular simplex with $v_1 + \dots + v_{d+1} = 0$. Then, $\lpos{A} = \R^d$ and for any $B \subset A$ of size at most $d,$ $\lpos{B} = \{0\}$.

Consider the case $d+1 < 2(k+1)$ in which $h(k,d) = 2(k+1)$. Set $A$ to be the set of vertices of the $(k+1)$-dimensional cross-polytope $\braces{\pm e_1, \dots, \pm e_{k+1}}$. Clearly, one has $\dim \lpos{A} = k+1$ but $\dim \lpos\parenth{A \setminus \braces{a}} \leq k$ for every $a \in A$.

\section{Nonhomogeneous Linear Systems}

\Href{Proposition}{prp:helly_hom_systems} was a byproduct of the proof method of the following result \cite[Theorem 1.2]{barany2024positive}:
\begin{prp}\label{prp:helly_cones}
Fix $k \in [d-1]$. Assume $\mathcal{F}$ is a finite family of convex sets in $\R^d$. If the intersection of any subfamily of at most $m(k,d)$ sets contains a $k$-dimensional cone, then the intersection $\bigcap\limits_{F \in \mathcal{F}} F$ of all sets in the family contains a $k$-dimensional cone.
\end{prp}
Let us show that \Href{Proposition}{prp:helly_cones} and \Href{Proposition}{prp:helly_hom_systems} are equivalent. Using a somewhat standard reduction procedure, one can equivalently assume that $\mathcal{F}$ is a finite family of the solution sets of finite nonhomogeneous systems of linear inequalities. The following trivial statement allows us to move a convex cone inside a convex set:
\begin{lem}\label{lem:apex_shift}
Let $K \subset \R^d$ be a closed convex set containing a cone $C$ with apex at $u$. Then $K$ contains the cone $v - u + C$ for every $v \in K$; $v$ is an apex of this cone.
\end{lem}

Using the fact that $m(k,d) \geq d+1$ in the classical Helly's theorem and \Href{Lemma}{lem:apex_shift}, we see that one can equivalently assume that $\mathcal{F}$ is a finite family of the solution sets of finite homogeneous systems of linear inequalities. Thus, \Href{Proposition}{prp:helly_cones} and \Href{Proposition}{prp:helly_hom_systems} are indeed equivalent.

Unfortunately, the previous argument does not apply directly in the colorful case. We can show the following weaker colorful version of \Href{Proposition}{prp:helly_cones}, which is essentially a nonhomogeneous version of \Href{Theorem}{thm:colorful_linear_system_to_big}.

\begin{thm}\label{thm:colorful_helly_cones}
Fix $k \in [d-1]$. Assume $\mathcal{F}_1, \dots, \mathcal{F}_{d+(d-k)+1}$ are finite families of convex sets in $\R^d$ such that the intersection of sets in any rainbow sub-selection $R$ of size at most $m(k,d)$ from these sets contains a $k$-dimensional cone. Then for some $i \in [d+(d-k)+1]$, all sets in $\mathcal{F}_i$ contain a translate of some $k$-dimensional cone $C$.
\end{thm}
Our proof strategy is to reduce the problem to the homogeneous case obtained in \Href{Theorem}{thm:colorful_linear_system_to_big}.

We need the following two properties of convex cones.

\begin{lem}\label{lem:cone_shift_intersection}
Assume $C$ is a closed convex cone in $\R^d$ with apex at the origin, and $c_1, \dots, c_n$ are points of $C$. Then
\[
c_1 + \dots + c_n + C \subset \bigcap\limits_{i \in [n]} \parenth{c_i + C}.
\]
\end{lem}
\begin{proof}
The lemma follows from \Href{Lemma}{lem:apex_shift} and the observation that $c_1 + \dots + c_n - c_i \in C$ for every $i \in [n]$.
\end{proof}

\begin{lem}\label{lem:cone_containment_set_bipolar}
Assume $K$ is a closed convex set in $\R^d$ and a non-trivial closed convex cone $C$ with apex at the origin is contained within $K^{\circ \circ}$. Then $K$ contains a translate of $C$.
\end{lem}
\begin{proof}
Since $K^{\circ \circ}$ contains a point different from the origin, $K$ is non-empty; say $p \in K$. Take any vector $v$ within $C$ and $t > 2$. Since $t v$ belongs to the closure of $\conv \parenth{K \cup \{0\}}$, there is a vector $v_t$ in $K$ such that the segment with endpoints $0$ and $v_t$ intersects the unit ball centered at $t v$, that is, $\lambda_t v_t = t v + \delta_t$ with $\enorm{\delta_t} \leq 1$ and $\lambda_t \in [0,1]$. By convexity, the point $p + \lambda (v_t - p)$ is in $K$ for all $\lambda \in [0,1]$. Thus,
\[
K \supset p + \frac{\lambda_t}{t} (v_t - p) = p \parenth{1 - \frac{\lambda_t}{t}} + \frac{t v + \delta_t}{t} \to p + v \quad \text{as} \quad t \to \infty.
\]
Hence, $p + v \in K$, and consequently, $p + C \subset K.$
\end{proof}

Now, we are ready to prove \Href{Theorem}{thm:colorful_helly_cones}, a non-homogeneous version of \Href{Theorem}{thm:colorful_linear_system_to_big}.

\begin{proof}[Proof of \Href{Theorem}{thm:colorful_helly_cones}]
The idea is to reduce the original families to families of closed polyhedral convex cones with apexes at the origin, that is, to the situation resolved in \Href{Theorem}{thm:colorful_linear_system_to_big}.

Without loss of generality, we can assume that all the sets are closed. Indeed, a convex set contains a non-trivial cone if and only if the closure of the set contains the closure of the cone. Also, we disregard the sets that coincide with the whole space $\R^d$.

For every rainbow sub-selection $R$ from the families $\mathcal{F}_{1}, \dots, \mathcal{F}_{d+(d-k)+1}$, we use $C(R)$ to denote a $k$-dimensional cone of the form $\pos\!\!\braces{v_1(R), \dots, v_k(R)}$ with linearly independent vectors $v_1(R), \dots, v_k(R)$ such that its translate is contained in the intersection set of $R$.

For every $i \in [d+(d-k)+1]$ and every $F_i \in \mathcal{F}_i$, define $F_i^\prime$ as the convex hull of the cones $C(R)$ for all the rainbow sub-selections $R$ of size at most $m(k,d)$ from $\mathcal{F}_{1}, \dots, \mathcal{F}_{d+(d-k)+1}$ containing $F_i$. Since $K$ and the origin are contained within $K^{\circ \circ}$, \Href{Lemma}{lem:apex_shift} implies that every such $C(R)$ is contained in $F_i^{\circ \circ}$. Hence, $F_i^\prime \subset F_i^{\circ \circ}$. Also, $F_i^\prime$ is a polyhedral cone with an apex at the origin, that is, $F_i^\prime$ is the solution set of a finite homogeneous system of linear inequalities $\iprod{a_i}{x} \leq 0$, $a_i \in A_i \subset \R^d \setminus \{0\}$.

Define $\mathcal{F}_1^\prime, \dots, \mathcal{F}_{d+(d-k)+1}^\prime$ by
\[
\mathcal{F}_i^\prime = \braces{F_i^\prime \st F_i \in \mathcal{F}_i}, \qquad i \in [d+(d-k)+1].
\]
On the one hand, we have that the intersection set of any rainbow sub-selection $R^\prime$ of size at most $m(k,d)$ from the families $\mathcal{F}_{1}^\prime, \dots, \mathcal{F}_{d+(d-k)+1}^\prime$ contains a $k$-dimensional cone with an apex at the origin. On the other hand, for every $i \in [d+(d-k)+1]$, the intersection set of $\mathcal{F}_i^\prime$ can be described as the solution set of a finite homogeneous system of linear inequalities $\iprod{a_i}{x} \leq 0$, $a_i \in \mathcal{A}_i \subset \R^d \setminus \{0\}$. The sets $\mathcal{A}_1, \dots, \mathcal{A}_{m(k,d)}$ satisfy the assumption of \Href{Theorem}{thm:colorful_linear_system_to_big}. Thus, for some $i \in [d+(d-k)+1]$, the system $\iprod{a}{x}\leq 0,\; a\in \mathcal{A}_i$, has $k$ linearly independent solutions, that is, the intersection $\bigcap\limits_{F_i^\prime \in \mathcal{F}_i^\prime} F_i^\prime$ contains a $k$-dimensional cone with an apex at the origin. As shown, this means that the intersection $\bigcap\limits_{F_i \in \mathcal{F}_i} F_i^{\circ \circ}$ contains a $k$-dimensional cone with an apex at the origin.

Applying \Href{Lemma}{lem:cone_containment_set_bipolar}, we see that each of the sets of $\mathcal{F}_i$ contains a translate of some $k$-dimensional cone. The proof of \Href{Theorem}{thm:colorful_helly_cones} is complete.
\end{proof}

We believe that, under the assumptions of \Href{Theorem}{thm:colorful_helly_cones}, there is a family whose intersection set contains a $k$-dimensional cone.

\end{document}